\documentclass[10pt]{amsart}
\usepackage{amssymb}
\usepackage{dsfont}
\usepackage{amscd}
\usepackage[mathscr]{euscript} 
\usepackage[all]{xy}
\usepackage{url}
\usepackage{stmaryrd}
\usepackage[retainorgcmds]{IEEEtrantools}
\usepackage{hyperref}
\title{A note on extending actions of infinitesimal group schemes}
\author[D. HOFFMANN]{Daniel Hoffmann$^{\dagger}$}
\address{$^{\dagger}$Instytut Matematyczny\\
Uniwersytet Wroc{\l}awski\\
Wroc{\l}aw\\
Poland}
\email{daniel.hoffmann@math.uni.wroc.pl}
\author[P. KOWALSKI]{Piotr Kowalski$^{\spadesuit}$}
\thanks{$^{\spadesuit}$Supported by NCN grant 2012/07/B/ST1/03513}
\address{$^{\spadesuit}$Instytut Matematyczny\\
Uniwersytet Wroc{\l}awski\\
Wroc{\l}aw\\
Poland}
 \email{pkowa@math.uni.wroc.pl} \urladdr{http://www.math.uni.wroc.pl/\textasciitilde pkowa/ }

 \DeclareMathOperator{\aut}{Aut} \DeclareMathOperator{\id}{id}
 \DeclareMathOperator{\fr}{Fr}

\DeclareMathOperator{\ch}{char}

\DeclareMathOperator{\arc}{Arc} 
\DeclareMathOperator{\coli}{\underrightarrow{\lim}}

\DeclareMathOperator{\ev}{ev}

\DeclareMathOperator{\htt}{ht}\DeclareMathOperator{\spec}{Spec}

\DeclareMathOperator{\evv}{ev}\DeclareMathOperator{\pgl}{PGL}
\DeclareMathOperator{\Der}{Der}

\newtheorem{theorem}{Theorem}[section]
\newtheorem{prop}[theorem]{Proposition}
\newtheorem{lemma}[theorem]{Lemma}
\newtheorem{cor}[theorem]{Corollary}

\theoremstyle{definition}
\newtheorem{definition}[theorem]{Definition}
\newtheorem{example}[theorem]{Example}
\newtheorem{remark}[theorem]{Remark}
\newtheorem{question}[theorem]{Question}

\begin{document}

\newcommand{\twoc}[3]{ {#1} \choose {{#2}|{#3}}}
\newcommand{\thrc}[4]{ {#1} \choose {{#2}|{#3}|{#4}}}
\newcommand{\Zz}{{\mathds{Z}}}
\newcommand{\Ff}{{\mathds{F}}}
\newcommand{\Cc}{{\mathds{C}}}
\newcommand{\Rr}{{\mathds{R}}}
\newcommand{\Nn}{{\mathds{N}}}
\newcommand{\Qq}{{\mathds{Q}}}
\newcommand{\Kk}{{\mathds{K}}}
\newcommand{\Pp}{{\mathds{P}}}
\newcommand{\ddd}{\mathrm{d}}
\newcommand{\Aa}{\mathds{A}}
\newcommand{\dlog}{\mathrm{ld}}
\newcommand{\ga}{\mathbb{G}_{\rm{a}}}
\newcommand{\gm}{\mathbb{G}_{\rm{m}}}
\newcommand{\gaf}{\widehat{\mathbb{G}}_{\rm{a}}}
\newcommand{\gmf}{\widehat{\mathbb{G}}_{\rm{m}}}

\maketitle
\begin{abstract}
We prove that iterative derivations on projective line cannot be expanded to iterative Hasse-Schmidt derivations, in the case when the iterativity rule is given by a non-algebraic formal group.
\end{abstract}

\section{Introduction}

This paper deals with the problem whether for a derivation $\partial$, there is a \emph{Hasse-Schmidt derivation} $(\partial_n)_{n\in \Nn}$ such that $\partial=\partial_1$. We recall that a Hasse-Schmidt derivation on a ring $R$ (see \cite{HS}) is a sequence
$$\mathbb{D}=(D_{i}:R\to R)_{i\in \Nn}$$
satisfying the following properties:
\begin{itemize}
\item $D_{0}=\id_R$,

\item each $D_{i}$ is additive,

\item for any $x,y\in R$ we have
$$D_{i}(xy)=\sum_{j+k=i}D_{j}(x)D_{k}(y).$$
\end{itemize}
A Hasse-Schmidt derivation $\mathbb{D}$ is \emph{iterative} if for all $i,j\in \Nn$ we have
$$D_i\circ D_j={i+j\choose i}D_{i+j}.$$

For a function $f:R\to R$ and a natural number $n$, we denote by $f^{(n)}$ the $n$-th compositional power of $f$. It is well-known that any derivation $\partial$ on a $\Qq$-algebra expands (uniquely) to an iterative Hasse-Schmidt derivation by the formula $\left(\frac{\partial^{(i)}}{i!}\right)_{i\in \Nn}$. Matsumura proved that the same (without uniqueness) is still true in the case of fields of positive characteristic $p$ for derivations satisfying the (necessary) condition $\partial^{(p)}=0$  \cite{Mats1}. In Matsumura's terminology, such a derivation $\partial$ is (\emph{strongly}) \emph{integrable}.

One may wonder why to consider the iterativity condition of such a specific form (although the characteristic $0$ example gives a rather strong motivation). It was noticed by Matsumura that actually the iterativity condition as above is governed by the additive group law $X+Y$. Since in characteristic $0$, any (one-dimensional) formal group is isomorphic to the additive one, it is a good choice indeed. However, in the case of positive characteristic there are many more formal group laws and it is an interesting question whether the corresponding derivations are integrable. A multiplicative version of Matsumura's theorem was proved by Tyc \cite{Tyc}, where the condition $\partial^{(p)}=0$ is replaced with the (necessary again) condition $\partial^{(p)}=\partial$.

In \cite{HK1}, the authors deal with a more general problem whether iterative \emph{$m$-truncated} ($m$ is a positive integer) Hasse-Schmidt derivations are integrable.  An iterative $m$-truncated Hasse-Schmidt derivation is a sequence $(\partial_i)_{i<p^m}$ satisfying the higher Leibnitz rules and the appropriate iterativity conditions, see \cite[Def. 2.11]{HK1}. Since a one-truncated additively iterative Hasse-Schmidt derivation is equivalent to a standard derivation $\partial$ satisfying the condition $\partial^{(p)}=0$ (similarly in the multiplicative case, where the necessary condition is $\partial^{(p)}=\partial$), this is a natural generalization. In \cite{HK1}, we extend the results of Matsumura and Tyc to the case of an arbitrary truncation (in the additive case, such a generalization is implicit in the work of Ziegler \cite{Zieg2}). A certain class of higher-dimensional commutative affine algebraic groups is treated by the first author in \cite{Hoff1}. In this paper, we focus on the one-dimensional case and we comment briefly on the higher-dimensional cases in Section \ref{highdim}.

We abbreviate the term ``additively (resp. multiplicatively) iterative Hasse-Schmidt derivation'' by ``$\ga$-derivation'' (resp. ``$\gm$-derivation''). Similarly, for any formal group law $F$, we have the notion of \emph{$F$-derivation} (see Definition \ref{defefiter}). In the truncated case, we have the corresponding notions of $\ga[m]$-derivations, $\gm[m]$-derivations and generally $F[m]$-derivations (see Definition \ref{deffitertrun}).

For other (than additive or multiplicative) formal group laws, the following question remains (as far as we know) open.
\begin{question}\label{mainq}
Let $F$ be a formal group law and $\mathbb{D}$ be an $F[m]$-derivation on a field $K$. Does $\mathbb{D}$ expand to an $F$-derivation on $K$?
\end{question}
More geometrically, the above question asks whether a (formal) group scheme action of a finite truncation of $F$ on a field extends to an action of the entire formal group (law) $F$ (see Section \ref{gpscheme1}). A natural test case is to consider the actions on a field of rational functions $K=k(t)$, equivalently on the generic point of projective line over a base field $k$. In this paper, we give a \emph{negative} answer to the question which is similar to Question \ref{mainq} and concerns actions on the entire projective line (rather than on its generic point). We show that for a formal group law $F$ which is neither isomorphic to the additive one nor to the multiplicative one, derivations on projective line cannot be expanded to $F$-derivations (Corollary \ref{lastcor}).

We would like to point out that for a derivation, being integrable (with a given iterativity rule) is related to the model-theoretic properties of \emph{existentially closed} fields with such derivations (see \cite{Zieg2}, \cite{K3} or \cite{HK}).

This paper is organized as follows. In Section \ref{gpscheme}, we recall the connection between Hasse-Schmidt derivations and group scheme actions. In Section \ref{secproj}, we prove the main theorem of this paper (Theorem \ref{htrtl}) about Hasse-Schmidt derivations on some algebraic varieties (including projective spaces). In Section \ref{sectrun}, we speculate about a possible approach to Question \ref{mainq}.

\section{Group scheme actions}\label{gpscheme}

In this section, we interpret the notion of a (truncated) Hasse-Schmidt derivation (abbreviated as \emph{HS-derivation} in the sequel) in terms of group scheme actions, recall the automorphism group functor and clarify the correspondence between the Lie algebra actions and the group scheme actions. We fix a field $k$ of positive characteristic $p$ and a $k$-algebra $R$. We also fix a complete local $k$-algebra $\mathcal{R}$ with the maximal ideal $I$ and a positive integer $m$.

\subsection{Definitions and examples of $F$-derivations and $f$-derivations}\label{secdefin}
For a more complete discussion of the notions introduced in this section, we refer the reader to Section 2.1 and Section 2.2 in \cite{HK1}.

We recall that a sequence of maps $(\partial_n:R\to R)_{n\in \Nn}$ is an HS-derivation on $R$ over $k$ (i.e. $\partial_n(k)=0$ for $n>0$) if and only if the corresponding map
$$\partial_X:R\to R\llbracket X\rrbracket,\ \ \ \partial_X(r)=\sum_{n=0}^{\infty}\partial_n(r)X^n$$
is a $k$-algebra homomorphism and $\partial_0$ is the identity map.

Let us fix $F\in k\llbracket X,Y\rrbracket$, a formal group law over $k$, that is $F$ satisfying
$$F(X,0)=X=F(0,X),\ \ \ F(F(X,Y),Z)=F(X,F(Y,Z)).$$
\begin{definition}\label{defefiter}
An HS-derivation $(\partial_n)_{n\in \Nn}$ on $R$ over $k$ is called \emph{$F$-iterative} if the following diagram is commutative
\begin{equation*}
 \xymatrix{R \ar[rr]^{\partial_Y} \ar[d]_{\partial_Z} & & R\llbracket Y\rrbracket \ar[d]^{\partial_X\llbracket Y\rrbracket} \\
	   R\llbracket Z\rrbracket \ar[rr]^{\ev_{F}} & & R\llbracket X,Y\rrbracket .}
\end{equation*}
\end{definition}
\noindent
We shorten the long phrase ``$F$-iterative HS-derivation over $k$'' to \emph{$F$-derivation}.

For a one-dimensional algebraic group $G$ over $k$ we use the term \emph{$G$-derivation} rather than $\widehat{G}$-derivation. In particular, a $\ga$-derivation is the same as an iterative HS-derivation.

We also consider a truncated version of the notion of an $F$-derivation. Let $m$ be a positive integer and $v_m,w_m,u_m$ denote the ``$m$-truncated variables'' i.e.
$$R[v_m]=R[X]/(X^{p^m}),\ \ R[v_m,w_m,u_m]=R[X,Y,Z]/(X^{p^m},Y^{p^m},Z^{p^m}).$$
For simplicity, we will denote $v_m,w_m,u_m$ by $v,w,u$.
Let $f\in k[v,w]$ be \emph{$m$-truncated group law} i.e. $f$ satisfies
$$f(v,0)=v=f(0,v),\ \ \ f(f(v,w),u)=f(v,f(w,u)).$$
\begin{definition}\label{deffitertrun}
An $m$-truncated HS-derivation $(\partial_n)_{n<p^m}$ over $k$ is called \emph{$f$-iterative} if the following diagram is commutative
\begin{equation*}
 \xymatrix{R \ar[rr]^{\partial_w} \ar[d]_{\partial_u} & & R[w] \ar[d]^{\partial_v[w]} \\
	   R[u] \ar[rr]^{\ev_f} & & R[v,w].}
\end{equation*}
\end{definition}
\noindent
We shorten the long phrase ``$f$-iterative $m$-truncated HS-derivation over $k$'' to \emph{$f$-derivation}.

We will briefly sketch the construction of a canonical $F$-derivation. The reader is referred to Section 3.2 of \cite{HK1} for more details. Let $t$ denote a variable which will play a somewhat different role than the variables $X,Y,Z$. There is an $F$-derivation $(\partial_n)_{n\in \Nn}$ on $R\llbracket t\rrbracket$ over $R$ such that $\partial_n(t)$ is the coefficient of $X^n$ in $F(t,X)$ (see \cite[Prop. 3.4]{HK1}). We call this $F$-derivation \emph{canonical $F$-derivation}. For example, if $F=X+Y$ and $\partial$ is the canonical $F$-derivation on $k\llbracket t\rrbracket$, then $\partial_1=\frac{\ddd}{\ddd t}$ is the usual derivative with respect to the variable $t$ and for any $n\in \Nn$ we have
$$ \partial_n\left(\sum\limits_{i=0}^{\infty}a_i t^i\right) = \sum\limits_{i=0}^{\infty}a_{i+n}{i+n\choose n}t^{i}.$$

\subsection{Bijective correspondences}\label{gpscheme1}
The category of \emph{affine group schemes} over $k$ is the category opposite to the category of \emph{Hopf algebras} over $k$ \cite[Section 1.4]{Water} (or it is the category of representable functors from $k$-algebras to groups, see \cite[Section 1.2]{Water}). A \emph{truncated group scheme} \cite{Chase2} over $k$ is an affine group scheme whose universe is of the form
$$\spec(k[v_{1,m_1},\ldots,v_{k,m_k}])=\spec\left(k[X_1,\ldots,X_k]/\left(X_1^{p^{m_1}},\ldots,X_k^{p^{m_k}}\right)\right).$$
By \cite[Section 14.4]{Water} such group schemes coincide with finite connected group schemes (called \emph{infinitesimal} in \cite{Demazure} and \emph{local} in \cite{Manin}). In this paper we mostly consider  truncated group scheme structures on $\spec(k[v_m])=\spec\left(k[X]/\left(X^{p^m}\right)\right)$ (the ``ordinary'' case). Any $m$-truncated group law $f$ naturally gives a truncated group scheme $G_f$ corresponding to the Hopf algebra with the following comultiplication map
$$\ev_{f(v_m\otimes 1,1\otimes v_m)}:k[v_m]\to k[v_m]\otimes k[v_m].$$
To go to the opposite direction, one needs to choose a generator of the $k$-algebra $k[v_m]$ (clearly, there is a natural choice).
\begin{remark}\label{notrunzero}
If $\ch(k)=0$, then by a theorem of Cartier \cite[Section 11.4]{Water} all Hopf algebras over $k$ are reduced, so there is no reasonable notion of a truncated group scheme.
\end{remark}

The category of \emph{formal groups} over $k$ is the category opposite to the category of \emph{complete Hopf algebras} over $k$ (or the category of representable functors from complete $k$-algebras to groups, see \cite[Chapter VII]{Hazew}). Similarly as in the truncated case, there is a correspondence between formal groups and formal group laws. Note that a truncated group scheme is both an affine group scheme and a formal group.

Let $f$ be a truncated group law over $k$ giving a truncated group scheme $G_f$. There is a natural correspondence between:
\begin{enumerate}
\item[(T1)] The set of all $f$-derivations on $R$.

\item[(T2)] The set of all $k$-group scheme actions of $G_f$ on $\spec(R)$.
\end{enumerate}
\noindent
Since a formal group law $F$ can be approximated by a direct system of truncated group laws $F[m]$ (for a more precise statement, see Section \ref{gpschemetrun}), we get a similar correspondence between:
\begin{enumerate}
\item[(F1)] The set of all $F$-derivations on $R$.

\item[(F2)] The set of all systems of compatible $k$-group scheme actions of $G_{F[m]}$ on $\spec(R)$ for $m>0$.
\end{enumerate}
\noindent
For an arbitrary $k$-scheme $V$ it makes sense to talk about $f$-derivations or $F$-derivations on $V$ using the conditions of type $(2)$ above.
\begin{remark}\label{arc}
Let $V$ be an arbitrary $k$-scheme.
\begin{enumerate}
\item[(i)] Since any HS-derivation on $R$ uniquely extends to the structure sheaf of $\spec(R)$, we get another equivalent condition (say, in the group scheme case and for an arbitrary $k$-scheme):
\begin{enumerate}
\item[(T2')] The set of all $f$-derivations on the structure sheaf of $V$.
\end{enumerate}

\item[(ii)] By the remark after Corollary 5.2 in \cite{KP5}, we have yet another equivalent condition:
\begin{enumerate}
\item[(T2'')] The set of all sections $V\to \arc_m(V)$ of the projection $\arc_m(V)\to V$ satisfying the commutative diagram as in the remark after Corollary 5.2 in \cite{KP5}.
\end{enumerate}
Note that $\arc_m$ is the $m$-th \emph{arc space} functor (see \cite{DeLo2000}) and $\arc_1$ coincides with the \emph{tangent space}.
\item[(iii)]  For a more general context (including and going beyond HS-derivations), the reader is advised to consult Example 2.12 in \cite{MS}.
\end{enumerate}
\end{remark}

Consider the \emph{automorphism functor} $A_{V/k}$, associated to a $k$-scheme $V$, from the category of $k$-schemes to the category of groups. For a $k$-scheme $W$, we have
$$A_{V/k}(W)=\aut_W(V\times_kW)$$
(see \cite[page 11]{Mao}). If this functor is representable we denote the representing $k$-group scheme by $\aut_k(V)$. Assume that
$A_{V/k}$ is representable. Then for the notion of an $f$-derivation in the case of a truncated group law $f$, we have a third equivalent condition, just like in the case of the usual group actions (using Yoneda Lemma), since there is always a natural map between the group functor represented by $G_f$ and the group functor $A_{V/k}$.
\begin{enumerate}
\item[(T3)] The set of all $k$-group scheme morphisms $G_f\to \aut_k(V)$.
\end{enumerate}
\noindent
Similarly in the case of a formal group law $F$:
\begin{enumerate}
\item[(F3)] The set of all compatible systems of $k$-group scheme morphisms
$$\left(G_{F[m]}\to \aut_k(V)\right)_{m>0}.$$
\end{enumerate}
\begin{remark}
We point out two cases (special cases of \cite[Theorem 3.7]{Mao}) when the functor $A_{V/k}$ is representable:
\begin{itemize}
\item $V$ is a projective variety over $k$;

\item $V=\spec(K)$ such that the extension $k\subseteq K$ is finite (including inseparable extensions!), see \cite[Example 3]{Mao}.
\end{itemize}
\end{remark}

\subsection{Truncations of group schemes}\label{gpschemetrun}
Let $G$ be an affine group scheme over $k$, $H$ the corresponding Hopf algebra, $\mathfrak{m}$ the kernel of the counit map $H\to k$ (the \emph{augmentation ideal}) and $m>0$. Using the base-change given by the automorphism $\fr^m:k\to k$,
we get the affine group scheme over $k$
$$G^{(m)}:=G\otimes_{(k,\fr^m)}k,$$
and a group scheme morphism $\fr^m_G:G\to G^{(m)}$. Let $G[m]$ be the kernel of $\fr^m_G$ which is a truncated $k$-group scheme. It corresponds to the quotient Hopf algebra $$H[m]:=H/\fr^m(\mathfrak{m})H.$$
We get a direct system of truncated $k$-group schemes $(G[m])_{m\in \Nn}$. If $G$ is an algebraic group over $k$, then $\coli(G[m])$ (the direct limit is taken in the category of formal groups) coincides with $\widehat{G}$, the formal group which is the formalization of $G$ (it follows from \cite[Lemma 1.1]{Manin}).

We abuse the language a little bit here identifying the formal group law $\widehat{G}$ (formalization of the algebraic group as in \cite[p. 13]{Manin}) with the corresponding formal group.

Similarly for a complete Hopf algebra $\mathcal{H}$, we have the analogous quotient $\mathcal{H}[m]$ which is again a Hopf algebra and also a complete Hopf algebra. Hence for a formal group $\mathcal{F}$, we have a direct system of truncated group schemes $\mathcal{F}[m]$ and in this case we get that $\mathcal{F}=\coli \mathcal{F}[m]$ (see \cite[Lemma 1.1]{Manin} again).

For any $m$-truncated group law $f$ and $l\leqslant m$ we have $G_f[l]\cong G_{f[l]}$ and similarly for any formal group law $F$ and the corresponding formal group $\mathcal{F}$, we have $G_{F[m]}\cong \mathcal{F}[m]$.
\subsection{Rational points of formal groups}\label{gpscheme2}
Let $\mathcal{F}$ be a formal group over $k$ corresponding to a complete Hopf algebra $\mathcal{H}$.
By $\mathcal{F}(\mathcal{R})$ (the set of ``$\mathcal{R}$-rational points'' of $\mathcal{F}$) we denote the set of all continuous $k$-algebra morphisms from $\mathcal{H}$ to $\mathcal{R}$. From general categorical reasons, $\mathcal{F}(\mathcal{R})$ is a group. If $\mathcal{F}$ comes from a formal group law $F$, then $\mathcal{F}(\mathcal{R})$ naturally corresponds to $I$, where $I$ is the maximal ideal of $\mathcal{R}$, and the operation on $\mathcal{F}(\mathcal{R})$ corresponds to
$$I\times I\ni (a,b)\mapsto F(a,b)\in I$$
(see \cite[Section 1.3]{Hazew}).

We note some properties of these groups.
\begin{lemma}\label{formalpoints}
Let $\alpha:\mathcal{F}\to \mathcal{F}'$ be a morphism of formal groups and $G$ an algebraic group over $k$. We have:
\begin{enumerate}
\item The induced map $\alpha:\mathcal{F}(\mathcal{R})\to \mathcal{F}'(\mathcal{R})$ is a homomorphism of groups.
\item There is a natural monomorphism of groups $\widehat{G}(\mathcal{R})\to G(\mathcal{R})$.
\end{enumerate}
\end{lemma}
\begin{proof}
Since the first part is clear, we comment only on $(2)$. For simplicity, we assume that $G=\spec(H)$ is affine. Then $\widehat{G}$ corresponds to $\widehat{H}$, where the completion is taken with respect to the augmentation ideal of the Hopf algebra $H$. Since the map $H\to \widehat{H}$ is a $k$-algebra map and any continuous homomorphism $\widehat{H}\to \mathcal{R}$ is determined by its values on $H$, we get a one-to-one map $\widehat{G}(\mathcal{R})\to G(\mathcal{R})$. An easy diagram chase shows that this map is a homomorphism of groups (see also Section 2.6 in \cite{K8}).
\end{proof}

\subsection{Truncated HS-derivations and actions of restricted Lie algebras}
In this subsection, we present the terminology from \cite{Tyc} and explain how does it fit into our context (see also (2.1b) in \cite{Chase2}). We assume that the field $k$ is perfect. Let $F$ be a formal group law. A function $f:k\llbracket X\rrbracket\to k\llbracket X\rrbracket$ is \emph{$F$-invariant} if for any $n\in\mathbb{N}$
$$f\circ D_n=D_n\circ f,$$
where $(D_n)_n$ is the canonical $F$-derivation on $k\llbracket X\rrbracket$ (see \cite[Section 3.2]{HK1}). The subset of $\Der_k(k\llbracket X\rrbracket)$ consisting of $F$-invariant derivations is a restricted Lie subalgebra, which is denoted by $L(F)$. It is a one-dimensional vector space over $k$ spanned by $D_1$ \cite[Lemma 2.1]{Tyc}. The restricted Lie algebra $L(F)$ \emph{acts} on a $k$-algebra $R$, if there is a homomorphism
$$\varphi:L(F)\to \Der_k(R)$$
of restricted Lie algebras.

It was noted in \cite{Tyc} that such an action is nothing else than a choice of $d\in\Der_k(R)$ such that
$d^{(p)}=c_F\cdot d$, where $c_F\in k$ is such that $D_1^{(p)}=c_F\cdot D_1$ (in this case the homomorphism $\varphi$
sends $D_1$ to $d$).

We need the following lemma which is a folklore.
\begin{lemma}\label{LFiso}
For any formal group law $F$ we have:
\begin{enumerate}
 \item the restricted Lie algebra $L(F)$ is isomorphic either to $L(X+Y)$ or to $L(X+Y+XY)$,
 \item the $1$-truncated group law $F[1]$ is isomorphic either to $\ga[1]$ or to $\gm[1]$.
\end{enumerate}
\end{lemma}
\begin{proof}
Item $(1)$ is an easy computation and item $(2)$ follows as in the beginning of the proof of Corollary \ref{lastcor}.
\end{proof}
The next result says that restricted Lie algebra actions correspond exactly to $1$-truncated group law actions.
\begin{prop}\label{LFactions}
 Let $d\in\Der_k(R)$. Then for any formal group law $F$ the following are equivalent:
 \begin{enumerate}
  \item the derivation $d$ defines an action of $L(F)$ on $R$,
  \item we have $d^{(p)}=c_F\cdot d$,
  \item there exists an $F[1]$-derivation $(d_n:R\to R)_{n<p}$ such that $d_1=d$.
 \end{enumerate}
\end{prop}
\begin{proof}
By Lemma \ref{LFiso}, we may assume that $F=X+Y$ or $F=X+Y+XY$. Equivalence of $(1)$ and $(2)$ was already explained above. Implication from $(3)$ to $(2)$ is clear.

For the remaining implication, the case of $F=X+Y$ is dealt with in \cite[Remark 2.9]{HK1}.
For the case $F=X+Y+XY$, we notice that if $\partial=(\partial_n)_{n<p}$ is an $F[1]$-derivation, then it satisfies (for $n\leqslant p-2$)
$$\partial_1\circ \partial_n=(n+1)\partial_{n+1}+n\partial_n.$$
Thus we can inductively define
$$d_{n+1}:=\frac{1}{n+1}(d\circ d_n-n d_n).$$
It is straightforward to show that $(d_{n})_{n<p}$ is an $F[1]$-derivation.
\end{proof}

\section{Iterative derivations on projective line}\label{secproj}

Assume that $k$ is a perfect field of characteristic $p>0$ and $m$ is a positive integer. In this section we discuss the problem whether truncated HS-derivations on projective line $\Pp^1=\Pp^1_k$ are integrable. A (truncated) HS-derivation on $\Pp^1$ is by definition a (truncated) HS-derivation on the structure sheaf of $\Pp^1$ (see \cite{KP5}).

It is very easy to describe (truncated) HS-derivations on $\Pp^1$: they correspond to (truncated) HS-derivations $\partial$ over $k$ on the polynomial algebra $k[t]$ such that $\partial$ preserves (after taking the unique extension to $k(t)$) the subalgebra $k[1/t]$. For a usual derivation $D$ the above condition means that $\deg(D(t))\leqslant 2$ which geometrically corresponds to the fact that the tangent bundle of projective line coincides with the line bundle $\mathcal{O}(2)$ (see Remark \ref{arc} for a more general geometric picture). Such an (truncated) HS-derivation on $\Pp^1$ satisfies a given iterativity condition if and only if the corresponding (truncated) HS-derivation on $k[t]$ does.

We note a projective version of our results about integrable derivations from \cite{HK1}.
\begin{theorem}\label{projintegr}
Let $\partial$ be a $\ga[m]$-derivation (resp. $\gm[m]$-derivation) on $\Pp^1$. Then $\partial$ can be expanded to a $\ga$-derivation (resp. $\gm$-derivation) on $\Pp^1$.
\end{theorem}
\begin{proof}
Let us fix $\partial$, a $\ga[m]$-derivation or a $\gm[m]$-derivation on $\Pp^1$. We also denote by $\partial$ the corresponding $m$-truncated HS-derivation on $k(t)$. By \cite[Prop. 4.5]{HK1} and \cite[Prop. 4.10]{HK1}, there is a canonical element $a\in k(t)$ for $\partial$. To finish the proof as in \cite[Thm. 4.7]{HK1}, it is enough to notice that the canonical $\ga$-derivation and the canonical $\gm$-derivation preserve $k[1/t]$. It is well-known in the additive case and in any case it follows by induction from the following formula ($n>0$)
$$0=\partial_n(tt^{-1})=t\partial_n(t^{-1})+\partial_1(t)\partial_{n-1}(t^{-1}),$$
since $\partial_1(t)=1$ or $\partial_1(t)=t+1$ and $\partial_i(t)=0$ for $i>1$.
\end{proof}
We will show that for most of the other iterativity rules, truncated HS-derivations on $\Pp^1$ cannot be integrated. It will follow from the result below saying that the existence of a non-trivial $F$-derivation on an algebraic variety from a certain class including projective spaces (see Remark \ref{hdim}(1)) is quite a restrictive condition on the formal group law $F$. For the notion of the height of a formal group law we refer the reader to \cite[Def. 18.3.3]{Hazew}.
\begin{theorem}\label{htrtl}
Let $F$ be a formal group law over $k$ and $W$ be a $k$-scheme such that the functor $A_{W/k}$ is representable by a linear algebraic group over $k$. If there is a non-trivial $F$-derivation on $W$, then $\htt(F)=1$ or $\htt(F)=\infty$.
\end{theorem}
\begin{proof}
Since the height of a formal group law does not change after a base-change, we can assume that $k$ is algebraically closed. Let $G$ be a linear algebraic group over $k$ which represents $A_{W/k}$ and $\mathcal{F}$ be the formal group corresponding to $F$. We denote by $H$ the Hopf algebra of $G$ and by $\mathcal{H}$ the complete Hopf algebra corresponding to $\mathcal{F}$ (so $\mathcal{H}$ is isomorphic to the power series ring in one variable $t$ and the complete Hopf algebra structure is given by $t\mapsto F(t\otimes 1+1\otimes t)$). By the bijective correspondences from Section \ref{gpscheme1}, we get a compatible system of $k$-group scheme morphisms corresponding to the non-trivial $F$-derivation on $W$
$$\left(\mathcal{F}[m]\to \aut_k(W)=G\right)_{m>0}.$$
For each $m$, the above morphism factors through a morphism
$$\mathcal{F}[m]\to G[m],$$
since the morphism $\mathcal{F}[m]\to G$ composed with $\fr^m_{G}$ is the trivial morphism  and
$$G[m]=\ker(\fr^m_{G})$$
(see Section \ref{gpschemetrun}).
Since we have
$$\coli(\mathcal{F}[m])\cong \mathcal{F},\ \ \ \ \coli(G[m])\cong \widehat{G}$$
(see Section \ref{gpschemetrun}), we get a non-trivial morphism of formal groups
$$\Psi:\mathcal{F}\to \widehat{G}.$$
We proceed to show that the ``Zariski closure of the image of $\Psi$'' is a commutative algebraic subgroup $V$ of $G$ and that $\Psi$ factors through a formal group morphism $\mathcal{F}\to \widehat{V}$.

Let  $\Phi$ be dual to $\Psi$, $\mathcal{I}$ the kernel of $\Phi$, $I=\mathcal{I}\cap H$ and $R$ the quotient of $H$ by $I$. We have a commutative diagram:
\begin{equation*}
  \xymatrix{I \ar[rr]^{} \ar[d]_{} & &\mathcal{I} \ar[d]_{}\\
  H  \ar[rr]^{} \ar[d]_{}  & &\widehat{H} \ar[d]^{\Phi} \\
  R   \ar[rr]^{} \ar[d]_{} &  &\mathcal{H}\\
  \widehat{R} \ar[rru]^{\alpha},}
\end{equation*}
where the completion of $R$ is taken with respect to the image of the maximal ideal in the local ring of $I\in G$. Since $\Phi$ is non-trivial (i.e. $\Phi(\widehat{H})\neq k$), the induced map $\alpha:\widehat{R}\to \mathcal{H}$ is non-trivial as well.

Let $V$ be the closed subvariety of $G$ corresponding to $\spec(R)$. The above diagram means that $\Phi$ factors through a morphism $\mathcal{F}\to \widehat{V}$ whose image is ``Zariski dense in $V$'' (since $R\to \mathcal{H}$ is one-to-one). We need to show that $V$ is a commutative algebraic subgroup of $G$.

Let $\mathcal{R}$ be a complete $k$-algebra which is DVR and $K$ the algebraic closure of its field of fractions. Clearly, $\mathcal{H}\widehat{\otimes}_k\mathcal{R}\cong \mathcal{R}\llbracket t\rrbracket$ as complete $\mathcal{R}$-algebras.  After base-change, we get the following commutative diagram
\begin{equation*}
  \xymatrix{H\otimes_k\mathcal{R}  \ar[rr]^{\beta} \ar[d]_{}  & &\widehat{H}\widehat{\otimes}_k\mathcal{R} \ar[d]^{\Phi_{\mathcal{R}}} \\
  R\otimes_k\mathcal{R}  \ar[rr]^{\gamma} \ar[rrd]_{\gamma_a} &  &\mathcal{H}\widehat{\otimes}_k\mathcal{R} \ar[d]^{\evv_a} \\
  & & \mathcal{R},\\
}
\end{equation*}
where $\gamma$ is still an embedding, since $\mathcal{R}$ is flat over $k$. For any $a$ belonging to the maximal ideal of $\mathcal{R}$, we have the evaluation map $\evv_a:\mathcal{R}\llbracket t\rrbracket\to \mathcal{R}$ and we denote the composition of this map with $\gamma$ by $\gamma_a\in V(\mathcal{R})$.

The set of all evaluation maps as above forms a commutative group $\mathcal{F}(\mathcal{R})$ (see Section \ref{gpscheme2}). By Lemma \ref{formalpoints},
$$\Gamma:=\beta^*(\Psi^*(\mathcal{F}(\mathcal{R})))$$
is a commutative subgroup of $G(\mathcal{R})$. Using the natural embedding of $V(\mathcal{R})$ into $G(\mathcal{R})$ we see that $\Gamma \subseteq V(\mathcal{R})$. However $\gamma$ is an embedding and for any $T\in \mathcal{R}\llbracket t\rrbracket$ there is an element $a$ in the maximal ideal of $\mathcal{R}$ such that $T(a)\neq 0$ (either $a=0$ or any element $a$ in the maximal ideal of $\mathcal{R}$ works). Hence for any $r\in R$ there is $\gamma_a\in \Gamma$ such that $\gamma_a(r)\neq 0$. It means that $\Gamma$ is Zariski dense in $V(K)$. It is well-known (see e.g. \cite[Chapter 1 \S 2]{borellinear}) that the Zariski closure of a commutative subgroup of an algebraic group is a commutative subgroup. Hence $V$ is a commutative algebraic subgroup of $G$ and $\Psi:\mathcal{F}\to \widehat{G}$ factors through a formal group morphism $\mathcal{F}\to \widehat{V}$.

Since (by e.g. \cite[Section 10.2]{borellinear}) any commutative linear algebraic group over $k$ has an algebraic composition series whose factors are isomorphic either to $\ga$ or to $\gm$  (it is true even for solvable groups using the Lie-Kolchin theorem \cite[Cor. 10.5]{borellinear}), we get a non-trivial formal group morphism of the form $\mathcal{F}\to \gaf$ or of the form $\mathcal{F}\to \gmf$. The classification of commutative formal groups up to isogeny \cite[Chapter II \S 4]{Manin} gives that $\mathcal{F}\cong \gaf$ or $\mathcal{F}\cong \gmf$, hence (see \cite[Theorem 18.5.1]{Hazew}) $\htt(F)=1$ or $\htt(F)=\infty$.
\end{proof}
\begin{remark}\label{hdim}
\begin{enumerate}
\item  By \cite[page 21]{Mao}, for each positive integer $n$, the automorphism functor $A_{\Pp^n/k}$ is representable by the linear algebraic group $\pgl(n+1,k)$. Therefore, Theorem \ref{htrtl} applies to $\Pp^n$.

\item Theorem \ref{htrtl} cannot be generalized to the case of a $k$-scheme $W$ such that $A_{W/k}$ is representable by an arbitrary (i.e. not necessarily linear) algebraic group. The reason for that is a result of Manin (see \cite[Theorem 4.1]{Manin} and Remark 1 on page 74 of \cite{Manin}) saying (in particular) that for an arbitrary one-dimensional formal group $\mathcal{F}$, there is an Abelian variety $A$ and a non-trivial formal homomorphism $\mathcal{F}\to \widehat{A}$. Such a formal homomorphism $\mathcal{F}\to \widehat{A}$ gives rise to a non-trivial $F$-derivation on $A$, where $F$ is the formal group law corresponding to $\mathcal{F}$.
    \end{enumerate}
\end{remark}
\begin{cor}\label{lastcor}
Let $F$ be a formal group law over $k$. If $\htt(F)\neq 1$ and $\htt(F)\neq \infty$, then there is an $F[1]$-derivation on $\Pp^1$ which does not expand to an $F$-derivation.
\end{cor}
\begin{proof} By \cite[Theorem p. 69]{Demazure}, there are (up to isomorphism) only two $k$-group scheme structures on $\spec(k[X]/(X^p))$:
$$\ga[1]:=\ker(\fr:\ga\to \ga),\ \ \ \gm[1]:=\ker(\fr:\gm\to \gm),$$
since the Dieudonn\'{e} module of such a group scheme coincides with $k$ (length 1), the Frobenius map is $0$ and the Vershiebung map is either $0$ or $1$. If $F$ is a formal group law over $k$ and $\htt(F)=1$, then $F[1]\cong \gm[1]$ (Vershiebung is $1$), otherwise $F[1]\cong \ga[1]$ (Vershiebung is $0$).

Let us assume that $\htt(F)\neq 1$ and $\htt(F)\neq \infty$. Therefore $F[1]\cong \ga[1]$ (since $\htt(F)\neq 1$). Since the standard derivation on $k[t]$ gives a non-zero derivation on $\Pp^1$, we have a non-zero $\ga[1]$-derivation $\partial$ on $\Pp^1$. Since $F[1]\cong \ga[1]$, $\partial$ is also an $F[1]$-derivation. By Remark \ref{hdim}(1), this derivation does not expand to an $F$-derivation.
\end{proof}
\begin{example}\label{ex0}
It is easy to see (reasoning as at the end of the proof of Theorem \ref{htrtl}) that for $F$ as in Corollary \ref{lastcor}, there is $m$ such that a non-zero $F[1]$-derivation on $\Pp^1$ does not expand to an $F[m]$-derivation on $\Pp^1$. One could wonder about the particular value of $m$. Considering $F$ and $\mathbb{D}:=(\partial_i)_{i<p^3}$ from Example \ref{canex} ($p=2$ there), it is easy to compute that
$$\partial_4( 1/t) = t^{10} + t^4 + t + t^{-2} + t^{-5},$$
so $k[1/t]$ is not preserved by $\mathbb{D}$, hence $\mathbb{D}$ does not extend from $\Aa^1$ to $\Pp^1$. This calculation suggests that $m=3$ (if $F[3]\ncong \ga[3]$ and $F[3]\ncong \gm[3]$).
\end{example}
\begin{remark}
Results of this section may look related to the isotriviality theorems from \cite{KP5}. However, there is a fundamental difference. In \cite{KP5}, a different type of HS-derivations on projective varieties is considered since their restriction to the base field $k$ is \emph{generic} (i.e. this restriction gives an \emph{existentially closed structure}, see e.g. \cite{Ho}). In this paper such a restriction is the 0-derivation, which is very non-generic. Clearly, the proofs are related since they both use the automorphism functor.
\end{remark}

\subsection{Higher dimensional iterativity conditions}\label{highdim}
In this subsection, we will briefly discuss the case of an action of higher-dimensional groups (different actions than those which were discussed in Remark \ref{hdim}). Let $e\geqslant 2$ and we fix an $e$-dimensional algebraic group $G$  such that
$$k(G)\cong_kk(X_1,\ldots,X_e).$$
We would like to check (as a natural starting point) whether an analogue of Theorem \ref{projintegr} is true in this context, i.e. whether the canonical $G$-derivation on $k(X_1,\ldots,X_e)$ extends to a $G$-derivation on $\Pp^e$. For simplicity, we will assume that $G=U$ is unipotent of dimension $2$. This case is analyzed in \cite{Hoff1} and the results about integrable $U$-derivations (for actions on field extensions) are proved there.

In the case of $\Pp^1$, the condition for an HS-derivation $\mathbb{D}$ on $k[X]$ to extend to $\Pp^1$ was very easy: it was enough that $\mathbb{D}$ (after extending to $k(X)$) preserves the subring $k[1/X]$. In the case of $\Pp^2$ we have two preservation conditions: an HS-derivation $\mathbb{D}$ on $k[X,Y]$ extends to $\Pp^2$ if and only if $\mathbb{D}$ (after extending to $k(X,Y)$) preserves the subrings $k[X/Y,1/Y]$ and $k[Y/X,1/Y]$. In principle, there is also one more ``cocycle condition'', but it is trivially satisfied, in the case of $\Pp^2$.

It is an easy computation to verify that both the subrings $k[X/Y,1/Y]$ and $k[Y/X,1/Y]$ are HS-subrings of $k(X,Y)$ with the canonical $\mathbb{G}_a^2$-derivation. We will see below that for $U$ which is a non-split extension of $\ga$ by $\ga$, the preservation condition fails.
\begin{example}
We take the unipotent group $U$ where the group law on $U$ is given by
 $$(X,Y)*(W,Z)=\left(X+W+\sum\limits_{i=1}^{p-1}\frac{1}{p}{p\choose i}Y^iZ^{p-i}, Y+Z \right).$$
We want to check whether $k[Y/X,1/X]$ is an HS-subring of $k(X,Y)$ equipped with the canonical $U$-derivation. However, for $p>3$ we have
 $$D_{(0,1)}(1/X)=-\frac{D_{(0,1)}(X)}{X^2}=-\frac{Y^{p-1}}{X^2} \not\in k[Y/X,1/X].$$
\end{example}

\section{Speculations on a possible answer to Question \ref{mainq}}\label{sectrun}

Motivated by Theorem \ref{htrtl}, we were trying to proceed towards a \emph{negative} answer to Question \ref{mainq}. In the previous version of this paper, we noted our (unsuccessful) attempts into this direction. However, while the paper was being refereed, some evidence appeared suggesting that the answer may be positive after all. Even when we still do not know what the answer is, we think that these developments were quite interesting and we decided to summarize them in this section.

If one looks at Matsumura's proof of integrability of $p$-nilpotent derivations \cite{Mats1} (and the proofs of other integrability results, see e.g. \cite{Tyc} and \cite{HK1}), then one sees that the proof has the following steps.
\begin{enumerate}
\item Show that the canonical $F$-derivation $(\partial_n)_{n\in \Nn}$ on $k((t))$ restricts to $k(t)$.

\item For an arbitrary $F[m]$-derivation $D$ on a field $K$, find an embedding
$$\left(k(t),(\partial_n)_{n<p^m}\right)\to (K,D).$$

\item Expand the canonical $F[m]$-derivation on $k(t)$ (regarded as an $F[m]$-subfield of $K$)  to the canonical $F$-derivation on $k(t)$ and then extend it to $K$.
\end{enumerate}
To carry on these steps, one needs to answer the following first (Step 1).
\begin{question}\label{qrest}
For which formal group laws $F$, the canonical $F$-derivation on $k\llbracket t\rrbracket$ restricts to $k[t]$ or (after extending the canonical $F$-derivation to the field of Laurent series) restricts to $k(t)$?
\end{question}
Obviously, the canonical $F$-derivation restricts to $k[t]$ if and only if  $F\in k[X]\llbracket Y\rrbracket$, but it was unclear to us for which formal group laws the latter condition holds. We have decided to test the formal group laws from \cite[(3.2.3)]{Hazew}, which are defined as follows. Let $h$ be a positive integer and
$$f_{\Delta_h}(X):=\sum_{n=0}^{\infty}\frac{1}{p^h}X^{p^{nh}}.$$
Then we get the following formal group law over $\Zz$
$$F_{\Delta_h}(X,Y):=f_{\Delta_h}^{-1}(f_{\Delta_h}(X)+f_{\Delta_h}(Y)).$$
This formal group law is obviously isomorphic (over $\Qq$) to $\gaf$. But its reduction modulo $p$, which we denote by $F_h(X,Y)$, is a formal group law over $\Ff_p$ of height $h$.
The first author has posted the following question on MathOverflow.
\begin{question}\label{quemo}
Does $F_h(X,Y)\in\mathbb{F}_p[X]\llbracket Y\rrbracket$?
\end{question}
To our surprise, Jonathan Lubin gave a positive answer \cite{LubinAns} to Question \ref{quemo} for $h\geqslant 2$. Moreover, at the same time Malkhaz Bakuradze prowed the same result in \cite{Bakur}.

Therefore, we obtain the following answer to Question \ref{qrest}.
\begin{theorem}\label{isorest}
Suppose that $k$ is an algebraically closed field and $F$ is a one-dimensional formal group over $k$. Then there is $\widetilde{F}$, a one-dimensional formal group over $k$, such that $F\cong \widetilde{F}$ and the canonical $\widetilde{F}$-derivation on $k\llbracket t\rrbracket$ restricts to $k[t]$.
\end{theorem}
\begin{proof}
If the height of $F$ is $1$ (resp. $\infty$), then $F\cong \gmf$ (resp. $F\cong \gaf$) and the theorem holds. For $F$ of finite height $h\geqslant 2$, we have $F\cong F_h$ (see \cite[Theorem 21.9.1]{Hazew}). By Lubin's answer \cite{LubinAns} to Question \ref{quemo}, the canonical $F_h$-derivation restricts to $k[t]$.
\end{proof}
Below we collect some observations about possible $F$-derivations on algebraic varieties. We are unable to say whether there is some kind of a general geometric picture emerging from them.
\begin{remark}
We assume that $k$ is an algebraically closed field. Let $V$ be an algebraic variety over $k$ and $F$ be a one-dimensional formal group law over $k$.
\begin{enumerate}
\item If $V=\Aa^1$, then there is a non-trivial $F$-derivation on $V$ (it follows from Theorem \ref{isorest}).

\item If $V=\Pp^1$ and if there is a non-trivial $F$-derivation on $V$, then $F\cong \gaf$ or  $F\cong \gmf$ (see Remark \ref{hdim}(1)).

\item There is an Abelian variety $V$ and a non-trivial $F$-derivation on $V$ (see Remark \ref{hdim}(2)).
\end{enumerate}
\end{remark}
We finish this section with a concrete example of an $F_2[3]$-derivation. This example is also needed for Example \ref{ex0}.
\begin{example}\label{canex}
It can be checked that the formulas below
$$\partial_1(t)=1,\ \partial_2(t)=t^2,\ \partial_3(t)=0,\ \partial_4(t)=t^6+t^{12},$$
$$\partial_5(t)=0,\ \partial_6(t)=t^4,\ \partial_7(t)=0$$
give an $F_2[3]$-derivation on $\Ff_2[t]$.
\end{example}

\bibliographystyle{plain}
\bibliography{harvard}

\end{document}